\documentclass{amsart}

\usepackage{colonequals}
\usepackage[alphabetic]{amsrefs}
\usepackage{enumitem}
\usepackage{amssymb}

\newtheorem{theorem}{Theorem}
\newtheorem{lemma}[theorem]{Lemma}

\theoremstyle{definition}

\newtheorem{example}[theorem]{Example}
\newtheorem{question}[theorem]{Question}

\theoremstyle{remark}
\newtheorem{remark}[theorem]{Remark}

\numberwithin{equation}{section}

\newcommand{\bP}{\operatorname{\mathbb{P}}}

\newcommand{\Ker}{\operatorname{Ker}}

\makeatletter
\@namedef{subjclassname@2020}{%
  \textup{2020} Mathematics Subject Classification}
  \makeatother

\begin{document}

\title{Linear families of smooth hypersurfaces over finitely generated fields}

\author{Shamil Asgarli}
\address{Department of Mathematics and Computer Science \\ Santa Clara University \\ 500 El Camino Real \\ USA 95053}
\email{sasgarli@scu.edu}

\author{Dragos Ghioca}
\address{Department of Mathematics, University of British Columbia, Vancouver, BC V6T 1Z2}
\email{dghioca@math.ubc.ca}

\author{Zinovy Reichstein}
\address{Department of Mathematics, University of British Columbia, Vancouver, BC V6T 1Z2}
\email{reichst@math.ubc.ca}


\begin{abstract}
Let $K$ be a finitely generated field. We construct an $n$-dimensional linear system $\mathcal{L}$ of hypersurfaces of degree $d$ in $\bP^n$ defined over $K$ such that each member of $\mathcal{L}$ defined over $K$ is smooth, under the hypothesis that the characteristic $p$ does not divide $\gcd(d, n+1)$ (in particular, there is no restriction when $K$ has characteristic $0$). Moreover, we exhibit a counterexample when $p$ divides $\gcd(d, n+1)$. 
\end{abstract}

\maketitle

\section{Introduction}\label{sect:intro}

The study of hypersurfaces varying in a pencil, or more generally, in a linear system of arbitrary dimension, is an active research area. For instance, determining the number of reducible members in a pencil is already a challenging problem \cite{Ste89}, \cite{Vis93}, \cite{PY08}. When the base field is a number field, the study of pencils has deep connections to Diophantine geometry; see, for example \cite{DGH21}. 
Linear systems of hypersurfaces over finite fields have been studied by Ballico~\cite{Bal07}, \cite{Bal09}. 

Our primary goal in the present paper is to address the following question from
a recent paper~\cite{AG22} by the first two authors. While the version stated in~\cite{AG22} 
was concerned with linear systems of hypersurfaces over finite fields, in this paper we will work over an arbitrary \emph{finitely generated field}. Recall that a field $K$ is called finitely generated if it is generated 
by a finite number of elements as a field (or equivalently, as a field extension of its prime subfield).

\begin{question}\label{quest:main}
Let $K$ be a finitely generated field and $r \geqslant 1$, $n\geqslant 2$, $d\geqslant 2$ be integers. Do there exist $r+1$ linearly independent homogeneous polynomials $F_0, F_1, ..., F_r \in K[x_0, \ldots, x_n]$ of degree $d$ such that the hypersurface
$$
X_{[a_0:a_1:\ldots:a_r]} = \{a_0 F_0 + a_1 F_1 + ... + a_r F_r = 0\} \subset \mathbb{P}^n
$$
is smooth for every $[a_0:a_1:\ldots:a_r]\in\mathbb{P}^r(K)$?
\end{question}

Here, as usual, ``smooth" means ``smooth at every $\overline{K}$-point", not just at every $K$-point.
Question~\ref{quest:main} can be rephrased in geometric terms as follows. Consider the linear system $\mathcal{L}=\langle F_0, ..., F_r\rangle$ of (projective) dimension $r$ spanned by $F_0, \ldots, F_r$. We say that $\mathcal{L}$ is $K$-\emph{smooth} if for every $[a_0:a_1:\ldots:a_r]\in \mathbb{P}^r(K)$, the hypersurface cut out by $a_0 F_0 + a_1 F_1 + ... + a_r F_r$ is smooth in $\mathbb{P}^n$. In other words, Question~\ref{quest:main} asks for existence of a $K$-smooth linear system $\mathcal{L}$ in $\mathbb P^n$ of prescribed degree and dimension.

We show that, under a mild assumption on the characteristic, the maximum value of $r$ for which Question~\ref{quest:main} has a positive answer is $r=n$.

\medskip 

\begin{theorem}~\label{thm:main}
Let $K$ be an arbitrary field. 
\begin{enumerate}
    \item \label{main-thm-item-1} If $r\geqslant n+1$, then there does not exist a $K$-smooth linear system of (projective) dimension $r$ (of any degree $d \geqslant 2$).

    \item \label{main-thm-item-2} Suppose $K$ is a finitely generated field of characteristic $p \geqslant 0$. If $r \leqslant n$ and $p\nmid \gcd(d, n+1)$, then there exist  homogeneous polynomials $F_0, \ldots, F_r$ in $x_0, \ldots, x_n$ of degree $d$ such that $\mathcal{L} = \langle F_0, \ldots, F_r \rangle$ is a $K$-smooth linear system of (projective) dimension $r$. 

\end{enumerate}
\end{theorem} 

Note that the assumption $p\nmid \gcd(d, p+1)$ on the characteristic of $K$ holds automatically when $\operatorname{char}(K)=0$. On the other hand, we will show in Section~\ref{sect:quadrics} that this assumption \emph{cannot} be dropped in general. More precisely, we will show that no $n$-dimensional linear system of degree $2$ hypersurfaces in $\mathbb P^n$ can be $K$-smooth in the case where $K$ is a field of characteristic $2$ and $n \geqslant 1$ is an odd integer; see~Theorem~\ref{thm:quadrics}.

The case where $r=1$, which corresponds to a pencil of hypersurfaces, is of particular interest. For any given $n$, the condition that $p\nmid \gcd(d, n+1)$ is satisfied for all but finitely many characteristics $p$. 
In particular, Theorem~\ref{main-thm-item-2} tells us that for every value of $d \geqslant 1$ 
and every finitely generated field $K$ there exists
\begin{itemize}
    \item 
    a $K$-smooth pencil of degree $d$ in $\mathbb{P}^2$ if $\operatorname{char}(K)\neq 3$.
    \item  
    a $K$-smooth pencil of degree $d$ in $\mathbb{P}^3$ if $\operatorname{char}(K)\neq 2$.
    \item 
    a $K$-smooth pencil of degree $d$ in $\mathbb{P}^4$ if $\operatorname{char}(K)\neq 5$.
     \item 
     a $K$-smooth pencil of degree $d$ in $\mathbb{P}^5$ if $\operatorname{char}(K)\neq 2, 3$.
\end{itemize}

On the other hand, the main result of~\cite{AG22}*{Theorem 1.3} proves the existence of a $K$-smooth pencil $\mathcal{L}$ of degree $d$ hypersurfaces in $\mathbb{P}^n$ defined over the field $K = \mathbb{F}_q$ under a different hypothesis:
\begin{align*}
q > \left(\frac{1+\sqrt{2}}{2}\right)^2 \left((n+1)(d-1)^{n}\right)^2 \left((n+1)(d-1)^{n}-1\right)^2\left((n+1)(d-1)^{n}-2\right)^2.
\end{align*}
In particular, an $\mathbb{F}_q$-smooth pencil of degree $d$ hypersurfaces exists in any characteristic as long as $q$ is sufficiently large. It is reasonable to ask if smooth pencils of every degree exist over every finitely generated field.

\medskip

\textbf{Acknowledgements.} In an earlier version of this paper our main result, Theorem~\ref{thm:main}\eqref{main-thm-item-2}, was only stated for finite fields. We are grateful to Angelo Vistoli for suggesting that it can be extended to finitely generated fields and contributing the inductive argument of Section~\ref{sect:induction}. 

The first author is supported by a postdoctoral research fellowship from the University of British Columbia and the NSERC PDF award. The second and third authors are supported by NSERC Discovery grants.

\section{Proof of Theorem~\ref{thm:main}\eqref{main-thm-item-1}}\label{sect:main-result}

In this section $K$ will denote an arbitrary field. We will denote by $K[x_0, \ldots, x_n]_d$ the space of homogeneous polynomials of degree $d$
in $x_0, \ldots, x_n$ with coefficients in $K$. This is a $K$-vector space of dimension $N=\binom{n+d}{d}$.
Points of the projective space $\mathbb{P}(K[x_0, \ldots, x_n]_d)$ are naturally identified with degree $d$ hypersurfaces in $\mathbb{P}^n$.

We proceed with the proof of part \eqref{main-thm-item-1} of Theorem~\ref{thm:main}. Assume the contrary: there exists a $K$-smooth linear system $\mathcal{L} \subset K[x_0, \ldots, x_n]_d$ of (affine) dimension $\geqslant n + 2$. 

Let $x_0^{d-1} K[x_0, \ldots, x_n]_1$ denote the ($n+1$)-dimensional $K$-vector space of degree $d$ forms
divisible by $x_0^{d-1}$. Any such form can be written as $x_0^{d-1} l(x_0, \ldots, x_n)$, where $l \in K[x_0, \ldots, x_n]_1$. Consider the $K$-linear map
$$
\Psi \colon K[x_0, \ldots, x_n]_d \to x_0^{d-1} K[x_0, \ldots, x_n]_1
$$
which removes from $F \in \mathcal{L}(K)$ all monomials which are not multiples of $x_0^{d-1}$. In other words, for any non-negative integers $i_0, \ldots, i_n$
satisfying $i_0 + \ldots + i_n = d$, 
\[  \Psi(x_0^{i_0} x_1^{i_1} \ldots x_n^{i_n}) =  \begin{cases} \text{$x_0^{i_0} x_1^{i_1} \ldots x_n^{i_n}$, if $i_0 \geqslant d - 1$, and} \\
                                                                 \text{$0$, otherwise}. \end{cases} \]
The kernel, $\Ker(\Psi)$, is precisely the set of polynomials $F \in K[x_0, \ldots, x_n]_d$ with the property that the associated hypersurface in $\mathbb{P}^n$ is singular 
at $P = [1: 0: \ldots : 0]$. Since the codimension of $\Ker(\Psi)$ in $K[x_0, \ldots, x_n]_d$ is at least $\dim(x_0^{d-1} K[x_0, \ldots, x_n]_1) = n + 1$
and $\dim(\mathcal{L}) \geqslant n + 2$, we see that $\mathcal{L}\cap \Ker(\Psi)$ must contain a non-zero $K$-point of $\mathcal{L}$. In other words, 
$\mathcal{L}(K)$ contains a hypersurface which is singular at $P$. This shows that $\mathcal{L}$ cannot be $K$-smooth.
\qed

\section{Proof of Theorem~\ref{thm:main}\eqref{main-thm-item-2} in the case, where $K$ is a finite field} 
\label{sect:finite}

We begin by exhibiting two families of smooth hypersurfaces of degree $d \geqslant 2$ over an arbitrary field $K$ of characteristic $p \geqslant 0$.

\begin{lemma}\label{lemma:smoothness-fermat}
Suppose $p\nmid d$. Set $F=c_0 x_0^{d}+ c_1 x_1^{d} + ... + c_n x_n^{d}$. If $c_0, c_1, \ldots, c_n \neq 0$, then
$F$ cuts out a smooth hypersurface in $\mathbb{P}^n$.
\end{lemma}

\begin{proof} This is clear from the Jacobian criterion: the equations 
\[ \frac{\partial F}{\partial x_i}= d c_i x_i^{d-1} = 0 \; \; \text{($i = 0, 1, \ldots, n$)} \]
have no common solution in $\mathbb{P}^n$.
\end{proof}

\begin{lemma}\label{lemma:smoothness-klein}
Suppose $p\mid d$ but $p\nmid (n+1)$. Set $F= c_0 x_0^{d-1} x_1 + c_1 x_1^{d-1} x_2 + ... + c_n x_n^{d-1} x_0$.
If $c_0, c_1, \ldots, c_n \neq 0$, then $F$ cuts out a smooth hypersurface in $\mathbb{P}^n$.
\end{lemma}

\begin{proof}
Assume the contrary: the hypersurface cut out by $F$ in $\mathbb P^n$ is singular at 
some point $P=[u_0:u_1:...:u_n] \in \mathbb{P}^n$. By symmetry we may assume without loss 
of generality that $u_1\ne 0$. Using the Jacobian criterion, and remembering that $p\mid d$, we obtain:
\begin{equation}\label{eq:smoothness-klein-partial}
\frac{\partial{F}}{\partial x_i}(P) = c_{i-1} u_{i-1}^{d-1} - c_i u_i^{d-2} u_{i+1}  = 0 
\end{equation}
for each $0\leqslant i\leqslant n$, where the subscripts are taken modulo $n+1$. Multiplying both sides of~\eqref{eq:smoothness-klein-partial} by $u_i$, we obtain
\begin{equation}\label{eq:smoothness-klein}
c_{i-1} u_{i-1}^{d-1} u_i = c_i u_i^{d-1} u_{i+1} .
\end{equation}
Now recall that
\[ F(P) = c_0 u_0^{d-1} u_1 + c_1 u_1^{d-1} u_2 + ... + c_n u_n^{d-1} u_0 = 0 . \]
By~\eqref{eq:smoothness-klein}, the $n$ terms in this sum are all equal to each other. Hence,
$$
0 = F(P) = \sum_{i=0}^{n} c_i u_i^{d-1} u_{i+1} = (n+1) c_0 u_0^{d-1} u_{1}. 
$$
Since $p\nmid (n+1)$, $c_0\neq 0$, and $u_1 \neq 0$, we conclude that $u_0=0$. 

\smallskip
We will divide the remainder of the proof into two cases, according to whether $d=2$ or $d\geqslant 3$. If $d \geqslant 3$, 
then~\eqref{eq:smoothness-klein-partial} tells us that $u_i = 0$ implies $u_{i-1} = 0$ for any $i \in \mathbb Z/(n+1) \mathbb Z$. (Recall that the subscripts in~\eqref{eq:smoothness-klein-partial} 
are viewed modulo $n + 1$.) Using this implication recursively, starting from $u_0 = 0$, we see that $u_0 = u_n = u_{n-1} = \ldots = u_1 = 0$, a contradiction.

\smallskip
Now assume $d=2$. In this case~\eqref{eq:smoothness-klein-partial} tells us that $u_{i - 1} = 0$ implies $u_{i+1} = 0$ for any $i \in \mathbb Z/(n+1) \mathbb Z$. Since we know that $u_0 = 0$, this tells us that $u_{i}=0$ for every even $i$. Since $d=2$, 
the assumption that $p$ divides $d$ tells us that $p = 2$ and the assumption that $p$ does not divide 
$n + 1$ tells us that that $n = 2k$ is even. Thus, $2k + 2 \equiv 1$ modulo $n + 1$ and hence, $0 = u_{2k+2}=u_1=0$, a contradiction.
\end{proof} 

We are now ready to prove Theorem~\ref{thm:main}\eqref{main-thm-item-2} in the case, where $K=\mathbb{F}_q$ is a finite field. 
Since any $K$-linear subspace of a $K$-smooth linear system is again $K$-smooth, we may assume without loss of generality that $r = n$. Note also that $p \nmid \gcd(d, n + 1)$ if and only if $p \nmid d$ or $p \nmid n + 1$.
Thus we may consider two cases.

\smallskip
{\bf Case 1:} $p\nmid d$. We will explicitly construct a linear system $\mathcal{L}$ of dimension $r=n$ with the desired property. By the normal basis theorem, we can find an element $\alpha\in \mathbb{F}_{q^{n+1}}$ such that $\alpha, \alpha^q, \alpha^{q^2}, ..., \alpha^{q^n}$ form an $\mathbb{F}_q$-basis for the $(n+1)$-dimensional vector space $\mathbb{F}_{q^{n+1}}$. Let 
\begin{align*}
    F_0 &= (\alpha x_0 + \alpha^q x_1 + \alpha^{q^2} x_2 + ... + \alpha^{q^i} x_i + ... + \alpha^{q^n}x_n)^{d} , \\ 
    F_1 &= (\alpha^q x_0 + \alpha^{q^2} x_1 + \alpha^{q^3} x_2 + ... + \alpha^{q^{i+1}} x_i + ... + \alpha x_n)^{d}, \\
    F_2 &= (\alpha^{q^2} x_0 + \alpha^{q^3} x_1 + \alpha^{q^4} x_2 + ... + \alpha^{q^{i+2}} x_i + ... + \alpha^q x_n)^{d}, \\
    &\vdots \\
    F_n &= (\alpha^{q^n} x_0 + \alpha^{q} x_1 + \alpha^{q^2} x_2 + ... + \alpha^{q^{i+n}} x_i + ... + \alpha^{q^{n-1}} x_n)^{d}. 
\end{align*}
Note that the polynomials $F_i$ are not defined over $\mathbb{F}_q$. However, the set $\{F_0, F_1, ..., F_n\}$ is invariant under the action of the $q$-th power Frobenius map. Thus, the linear system $\mathcal{L}=\langle F_0, ..., F_n\rangle$ is defined over $\mathbb{F}_q$, that is, one can find a set of new generators $G_0, G_1, \ldots, G_n$ for $\mathcal{L}$ where \emph{each} $G_i$ is defined over $\mathbb{F}_q$. 

We claim that $F_0, F_1, ..., F_n$ are linearly independent over $\overline{\mathbb{F}_q}$. To prove this claim, let
\begin{equation}\label{eq:new-coord}
    y_j = \alpha^{q^j} x_0 + \alpha^{q^{j+1}} x_1 + \alpha^{q^{j+2}} x_2 + ... + \alpha^{q^{j+i}} x_i + ... + \alpha^{q^{j+n}}x_n
\end{equation}
for each $0\leqslant j\leqslant n$, and observe that $F_i=y_i^d$. The linear map $x_i\mapsto y_i$ is a linear automorphism of $\mathbb{P}^n$. Indeed, the matrix of this linear transformation, known as a \emph{Moore matrix}, is non-singular;
see, e.g., \cite{Go96}*{Corollary 1.3.4}.
Thus, $y_0, \ldots, y_n$ are algebraically independent over $\mathbb{F}_q$ 
and hence, over $\overline{\mathbb{F}_q}$. Consequently, $F_0, F_1, \ldots, F_n$ are linearly independent over $\overline{\mathbb F_q}$. This proves the claim. In summary,
$\mathcal{L}=\langle F_0, F_1, \ldots, F_n\rangle$ is a linear system of degree $d$ hypersurfaces in $\mathbb{P}^n$ defined over $\mathbb F_q$ of (projective) dimension $r=n$. 

It remains to show that $\mathcal{L}$ is $\mathbb F_q$-smooth.
Indeed, suppose 
\begin{equation}\label{eq:checking-singularity}
X = \{ c_0 F_0 + c_1 F_1 + ... + c_n F_n = 0 \}
\end{equation}
is a singular hypersurface $X$ which belongs to $\mathcal{L}$,
for some $c_i\in\overline{\mathbb{F}_q}$ where not all $c_i$ are zero. 
Our goal is to show that $X$ is not defined over $\mathbb F_q$.
In the new coordinates $y_i$, we can express \eqref{eq:checking-singularity} as:
\begin{equation*}
X=\{c_0 y_0^d + c_1 y_1^d + ... + c_n y_n^d = 0\}.
\end{equation*}
Since $X$ is singular, we can apply Lemma~\ref{lemma:smoothness-fermat} to deduce that $c_i=0$ for some $i$. Without loss of generality, we may assume that $c_0=0$. By applying the Frobenius map, we see that $X$ is sent to:
\begin{equation*}
X^{\sigma} = \{c_1^q F_2 + ... + c_n^q F_0 = 0 \}.
\end{equation*}
We claim that $X$ and $X^{\sigma}$ are distinct. Indeed, their defining equations are not multiples of one another: otherwise, there would exist a nonzero constant $b \in \overline{\mathbb F_q}$ such that $c_i^{q}= b\cdot c_{i+1}$ for each $0\leqslant i \leqslant n$ taken modulo $n+1$. As $c_0=0$, this would force $c_i=0$ for each $0\leqslant i\leqslant n$, which is a contradiction. Thus, $X$ is not defined over $\mathbb{F}_q$, as desired. We conclude that the linear system $\mathcal{L}$ is $\mathbb F_q$-smooth.

\smallskip
{\bf Case 2:} $p\mid d$ but $p\nmid (n+1)$. Define $y_0, \ldots, y_n$ by the formula 
$\eqref{eq:new-coord}$, and set $F_i=y_i^{q} y_{i+1}$ for $0\leqslant i\leqslant n-1$ and $F_n=y_n^{q} y_{0}$. 
Arguing as in Case 1, one readily checks that $\mathcal{L}=\langle F_0, F_1, ..., F_n\rangle$ 
is a linear subspace of (projective) dimension $n$ defined over $\mathbb F_q$.
Moreover, the same argument as in Case 1, with Lemma~\ref{lemma:smoothness-klein} used 
in place of Lemma~\ref{lemma:smoothness-fermat}, shows that $\mathcal{L}$ is $\mathbb{F}_q$-smooth. 

\smallskip
This completes the proof of
Theorem~\ref{thm:main}\eqref{main-thm-item-2} in the case, where $K=\mathbb{F}_q$ is a finite field. \qed

\section{Conclusion of the proof of Theorem~\ref{thm:main}\eqref{main-thm-item-2}}
\label{sect:induction}

Given a finitely generated field $K$, we define its dimension $\dim(K)$ to be the Krull  dimension of any finitely generated $\mathbb{Z}$-algebra whose fraction field is $K$. In other words, $\dim(K)=\operatorname{tr deg}_{\mathbb{F}_p}(K)$ if $\operatorname{char}(K)=p>0$ and $\dim(K)=1+\operatorname{tr deg}_{\mathbb{Q}}(K)$ if $\operatorname{char}(K)=0$. 
In this section we will prove Theorem~\ref{thm:main}\eqref{main-thm-item-2} over an arbitrary finitely generated field $K$
by induction on $\dim(K)$. 
The inductive step will be based on the following lemma.

\begin{lemma}\label{lem:lifting}
Let $R$ be discrete valuation ring with fraction field $K$ and residue field $L$, and let $F_0, \ldots, F_r\in L[x_0, ..., x_n]$ be linearly independent homogeneous polynomials of degree $d$. Denote their liftings to $R$ by $\overline{F_0}, \ldots, \overline{F_r}\in R[x_0, ..., x_n]\subset K[x_0, ..., x_n]$, respectively. If the linear system $\langle F_0, \ldots, F_r \rangle$ is $L$-smooth, then the linear system $\langle \overline{F_0}, \ldots, \overline{F_r} \rangle$ is $K$-smooth.
\end{lemma}

\begin{proof}
Let $(a_0, \ldots, a_r)$ be in $K^{r+1}\setminus \{(0, \ldots, 0)\}$. We will show that the hypersurface in $\mathbb{P}^{n}_{K}$ defined by the form $a_0\overline{F_0}+\ldots + a_r \overline{F_r}$ is smooth. By scaling the $a_i$, we may assume that $a_i\in R$ for all $i$ and $a_i$ is invertible in $R$ for at least one $i$. Consider the hypersurface $X\subset \mathbb{P}^n_{K}$ defined by $a_0 \overline{F_0}+\cdots +a_r \overline{F_r}=0$. Then $X$ is flat over $\operatorname{Spec}(R)$ and its fiber over $\mathcal{L}$ is smooth by hypothesis. Since the smooth locus of the projection $X\to\operatorname{Spec}(R)$ is open in $X$, its complement must be empty. It follows that the fiber over the generic point of $\operatorname{Spec}(R)$ is smooth, as desired.
\end{proof}

We are now ready to finish the proof of Theorem~\ref{thm:main}\eqref{main-thm-item-2} by induction on the dimension of 
the finitely generated field $K$. If $\operatorname{dim}(K)=0$, then $K$ is a finite field. In this case Theorem~\ref{thm:main}\eqref{main-thm-item-2} is proved in Section~\ref{sect:finite}. If $\operatorname{dim}(K)>0$, then it is easy to see that $K$ admits a discrete valuation with finitely generated residue field $L$ such that $\dim(L)=\dim(K)-1$. Furthermore, if $\operatorname{char}(K)=0$, then this valuation can be chosen so that $\operatorname{char}(L)$ is positive and arbitrarily large. By applying Lemma~\ref{lem:lifting}, we can lift an $L$-smooth linear system of hypersurfaces 
in $\mathbb P^n$ to a $K$-smooth linear system of hypersurfaces in $\mathbb P^n$ of the same degree degree and the same dimension.
\qed

\section{Quadrics in characteristic $2$}\label{sect:quadrics}

In this section, we will show that the hypothesis $p\nmid \gcd(d, n+1)$ in our main theorem cannot be removed in general. We will focus on the case, where $p=d=2$ and $n$ is odd. Our goal is to prove the following result.

\begin{theorem}\label{thm:quadrics}
Suppose $n$ is an odd positive integer, and $K$ be a field of characteristic $2$ (not necessarily finitely generated). Then
for any $d \geqslant 2$ there does \textbf{not} exist a linear system $\mathcal{L} = \langle F_0, \ldots, F_n \rangle \subset K[x_0, \ldots, x_n]_2$ 
of (projective) dimension $n$ over $K$ such that each $K$-member of $\mathcal{L}$ is a smooth quadric hypersurface in $\mathbb{P}^n$.
\end{theorem}


\medskip 

We begin with the following lemma.

\begin{lemma}\label{lemma:sing-quadrics}
Let $K$ be a field of characteristic $2$ and $n \geqslant 1$ be an odd integer. 
Consider a quadric hypersurface $X\subset \mathbb{P}^n$ cut out by
$$
F(x_0, \ldots, x_n) = x_0^2 + G(x_1, x_2, ..., x_n)
$$
where $G \in K[x_1, \ldots, x_n]$ is a homogeneous polynomial of degree $2$. Then $X$ is singular.
\end{lemma}

\begin{proof} 
The Jacobian criterion gives rise to a homogeneous system
\[ \frac{\partial G}{\partial x_1} = \ldots = \frac{\partial G}{\partial x_n} =  0 \]
of $n$ linear equations in $x_1, \ldots, x_n$. (Note $x_0$ never appears in this system.)
We claim that this homogeneous linear system has a nontrivial solution. 
To prove the claim, it suffices to show that the matrix $M$ of this linear system is singular. Note that $M$ is the Hessian matrix of $G$ and hence, is symmetric. (Since $G$ is a quadratic polynomial, the entries of the Hessian matrix are constant.) 
Because we are in characteristic $2$, $M$ is also skew-symmetric. It remains to show
that a skew-symmetric square $n \times n$ matrix $M$ over any commutative ring has zero determinant, when $n$ is odd.

Indeed, consider the universal skew-symmetric matrix $n \times n$ matrix $A$ over the polynomial 
ring $R = \mathbb{Z}[x_{ij} | 1 \leqslant i < j \leqslant n]$. By definition, the
$(i, j)$-th entry of $A$ is $x_{ij}$ if $i < j$, $0$ if $i = j$ and $- x_{ij}$ if $i > j$. 
Taking the determinant on both sides of $A^T = - A$, and remembering that $n$ is odd, we obtain $\det(A) = -
\det(A)$ in $R$. Since $R$ is an integral domain of characteristic $0$, this implies that $\det(A) = 0$. 
A simple specialization argument (specializing $x_{ij}$ to the $(i, j)$-th entry of $M$) now shows that $\det(M) = 0$, as desired.

Thus, we have found $(0, \ldots, 0) \neq (t_1, \ldots, t_n) \in K^n$ 
such that for any point $P \in \mathbb{P}^n$ of the form $P=[t_0:\ldots:t_n]$, we have
\begin{equation} \label{e.jacobian}
\frac{\partial F}{\partial x_0}(P)= \ldots = \frac{\partial{F}}{\partial x_n}(P) = 0.
\end{equation}
Note that since $\deg(F)$ is even and we are in characteristic $2$, conditions~\eqref{e.jacobian}
do not guarantee that $F(P) = 0$. On the other hand, the partial derivatives of $F(x_0, \ldots, x_n)$ 
depend only on $x_1, \ldots, x_n$ and not on $x_0$.
We thus want to choose $t_0$ so that the resulting point $P = [t_0: \ldots : t_n]$ lies 
on the hypersurface $X$ cut out by $F$. To achieve this goal, we choose $t_0 \in \overline{K}$ so that
$$
t_0^2 = - G(t_1, t_2, ..., t_n).
$$
Then $P = [t_0:\ldots:t_n]\in\mathbb{P}^n(\overline{K})$ satisfies both \eqref{e.jacobian} and $F(P)=0$.
In other words, $X$ is singular at $P$.
\end{proof}

\begin{remark} If $K$ is a perfect field of characteristic $2$, 
then the above construction gives rise to
a singular point $P = [t_0: \ldots: t_n]$ of $X$ defined over $K$. 
Indeed, since $K$ is closed under taking square roots, we can always choose $t_0 \in K$ in the last step.
\end{remark}

\begin{remark} The conclusion of Lemma~\ref{lemma:sing-quadrics} is false when $n=2k$ is even. Indeed, the quadric hypersurface in $\mathbb{P}^n$ defined by the polynomial
$$
x_0^2 + x_1x_2+x_3x_4 + ... + x_{2k-1}x_{2k}=0
$$
is smooth.
\end{remark}

We now proceed with a proof of Theorem~\ref{thm:quadrics}.

\begin{proof}[Proof of Theorem~\ref{thm:quadrics}] Suppose, to the contrary, that $\mathcal{L}=\langle F_0, \ldots, F_n\rangle$ is a $K$-smooth linear system of quadric hypersurfaces of (projective) dimension $n$. Let $\mathcal{L}(K)$ denote the set of $K$-members of the system.

Consider the $K$-linear map
$$
\Psi: K[x_0, \ldots, x_n]_2 \to x_0 K[x_0, \ldots, x_n]_1
$$
introduced in Section~\ref{sect:main-result} (with $d = 2$). Recall that
$x_0 K[x_0, \ldots, x_n]_1$ denotes the $(n+1)$-dimensional $K$-vector space of quadratic forms in $x_0, \ldots, x_n$
divisible by $x_0$ and that $\Psi$ removes from $F \in K[x_0, \ldots, x_n]$ all monomials which are not multiples of $x_0$. 
When $d = 2$, the map $\Psi$ is given by the simple formula 
\[ (\Psi F)(x_0, \ldots, x_n) = F(x_0, x_1, \ldots, x_n) - F(0, x_1, \ldots, x_n). \]
As we noted in Section~\ref{sect:main-result}, $F$ lies in the kernel of $\Psi$ if and only if the hypersurface in $\mathbb{P}^n$ cut out by $F$
is singular at 
the point $[1:0: \ldots: 0]$. Since the linear system $\mathcal{L}$ is $K$-smooth, this tells us that the restricted map
$$
\Psi: \mathcal{L}(K) \to x_0 K[x_0, \ldots, x_n]_1
$$
is injective. Since the vector spaces $\mathcal{L}(K)$ and $x_0 K[x_0, \ldots, x_n]_1$ are of the same dimension $n + 1$, we conclude that
 $\Psi$ must also be surjective. In particular, there exists some $F\in\mathcal{L}(K)$ whose image under $\Psi$ is $x_0^2$. In other words, 
 \[ F (x_0, \ldots, x_n)  = x_0^2 +G(x_1, ..., x_n) \]
 for some quadratic form $G$ in $x_1, \ldots, x_n$. By Lemma~\ref{lemma:sing-quadrics}, $F$ cuts out a singular quadric hypersurface. This contradicts the assumption that each $K$-member of $\mathcal{L}$ is smooth. We conclude that a $K$-smooth linear system $\mathcal{L}$ of quadric hypersurfaces in $\mathbb P^n$ of dimension $n$ does not exist. 
\end{proof}

We have shown that the hypothesis $p\nmid \gcd(d, n+1)$ of Theorem~\ref{thm:main}\eqref{main-thm-item-2} cannot be removed in the case $p=2$. We do not know whether 
this assumption can be dropped for other primes $p$. We finish the paper with an example, 
which shows that it can be for one particular choice of $K$, $p$, $d$, and $n$.

\begin{example} Set $d=3$ and $n=2$ and consider the following cubic homogeneous polynomials with coefficients in $K = \mathbb{F}_3$:
\begin{align*}
    F_0 &= x^{3}+x^{2}y-x y^{2}+y^{3}+x^{2} z+x y z+y^{2} z-x z^{2}+z^{3} \\ 
    F_1 &= x^{3}+x^{2} y-x^{2}z-x y z+y^{2} z+z^{3} \\ 
    F_2 &= x^{3}-x^{2}y+x y^{2}+y^{3}+x^{2}z+x y z+y^{2} z-y z^{2}
\end{align*}
A computer calculation shows that $aF_0+bF_1+cF_2=0$ defines a smooth plane curve for each of the possible $3^2+3+1=13$ choices $[a:b:c]\in\mathbb{P}^2(\mathbb{F}_3)$. In other words, $\langle F_0, F_1, F_2 \rangle$ is a $\mathbb{F}_3$-smooth linear system of (projective) dimension $n=2$. Thus, the conclusion of Theorem~\ref{thm:main}\eqref{main-thm-item-2} holds in this example, 
even though $p$ divides $\gcd(d,n+1)$.
\end{example}

\end{document}